\documentclass[reqno]{amsart}
\usepackage{amsmath, amsthm, amssymb}   
\usepackage{graphicx}
\usepackage{caption}
\usepackage{subcaption}
\usepackage{mathrsfs}                   
\usepackage{mathdots}
\usepackage{color}
\usepackage{doi}
\usepackage{cleveref}
\usepackage{textcomp}
\usepackage{hyphenat} 
 \usepackage{csquotes}
  \usepackage[british]{babel}
\usepackage{footmisc}
 
\newcommand{\bA}{{\ensuremath{\mathscr{A}}}}
\newcommand{\bL}{{\ensuremath{\mathcal{L}}}}

\newcommand{\dL}{{\ensuremath{\mathscr{L}}}}
\newcommand{\bk}{{\ensuremath{\xi}}}
\newcommand{\bM}{{\ensuremath{\mathcal{M}}}}

 \newtheorem{theorem}{Theorem}
\newtheorem{lemma}{Lemma}
\newtheorem{assumption}{Assumption}
\theoremstyle{remark}
\newtheorem{remark}{Remark}
 
\begin{document}
\title[On optimal control of reflected diffusions]{On  optimal  control of reflected diffusions}
\author{Adam Jonsson}
\address{Department of Engineering Sciences and Mathematics \\Lule{\aa} University of Technology, Sweden}
\email{adam.jonsson@ltu.se}
\begin{abstract}
We study a simple singular control problem for a Brownian motion with constant drift and variance reflected at the origin. Exerting control  pushes the process towards the origin and generates a concave increasing state-dependent yield which is  discounted at a fixed rate. The most interesting feature of the problem is that its solution can be more complicated than antici\-pated. Indeed, for some parameter values, the optimal policy   involves two reflecting barriers and one repelling boundary, the action region being the union of two disjoint intervals.  We also show that the apparent ano\-maly  can be understood as involving a switch between two  strategies with different risk profiles: The risk-neutral decision maker initially gambles on  the more risky strategy, but lowers risk if this strategy underperforms.  
 \end{abstract}
 \maketitle 
 \section{Introduction}
This paper explores an unanticipated and, as of yet, unknown feature of a class of optimal control problems for reflected diffusions. Our focus centers on the process  $X=\{X_t\}_{t\geq 0}$   defined as  
\begin{align}\label{eq: X}
X_t=x+\mu t+\sigma  W_t-{\bk}_t+\dL_t, \quad t\geq 0,
 \end{align}
where $x \geq 0$ is an arbitrary initial condition, $\mu$ and $\sigma$ are constants, $W$  
 is a standard Brownian motion, and  where   ${\bk}$, the control process, is right-continuous, non-decreasing and non-anticipative (i.e., independent of future increments of $W$). Given $W$ and ${\bk}$, the non-decreasing process  $\dL$ is constructed so as to enforce a lower reflecting barrier at $0$. For   the exi\-stence and characteristics of such a process, see \cite[\S\,1]{CKM80} or  \cite[\S\,3]{EK88}. 
 
 Our problem can be interpreted in different ways (discussed below), but  we are especially interested in the case when   
$X$ models the water level in a dam having a reflecting lower boundary (cf. \cite{AB82,Yeh85,Zei04}).  In the absence of  control, the water level   evolves according to   reflected Brownian motion with constant drift and variance. Extracting an amount $\varepsilon$  of water causes the water level to drop by the same amount and generates a yield of $\eta(X)\varepsilon$.  Assume, as in \cite{Zei04}, that the extraction costs  increase  with depth in such a way that    $\eta(x)=(1-\frac{1}{x})^+$.  Thus,  $\eta$  vanishes in a neighborhood  of zero and    is    concave increasing on its support.  The  problem (stated  formally in  \Cref{sect: problem}) is to maximize the expected    total 
yield over an infinite time horizon for an exogenously given discount rate of $r>0$.  

As we show in this paper, despite the problem's simplicity, its solution can be quite complicated.  One might anticipate that control should be exerted when the state process exceeds a certain threshold $b>0$. Such  a control policy serves to enforce an \emph{upper}  reflecting barrier at $b$. For  most values of the three parameters ($\mu, \sigma$, $r$), the optimal policy indeed takes this simple form. 
Yet  for  a non-empty subset of the parameter space, the optimal policy is a \enquote{band policy} with three boundaries, $b< \theta< \lambda$. A  barrier is enforced at $\lambda$ until  $X$ drops to $\theta$, and a barrier is  enforced at $b$ from that point onward (see \Cref{fig: Fig1}).

 \begin{figure}
     \centering
       \includegraphics[width=0.6\textwidth]{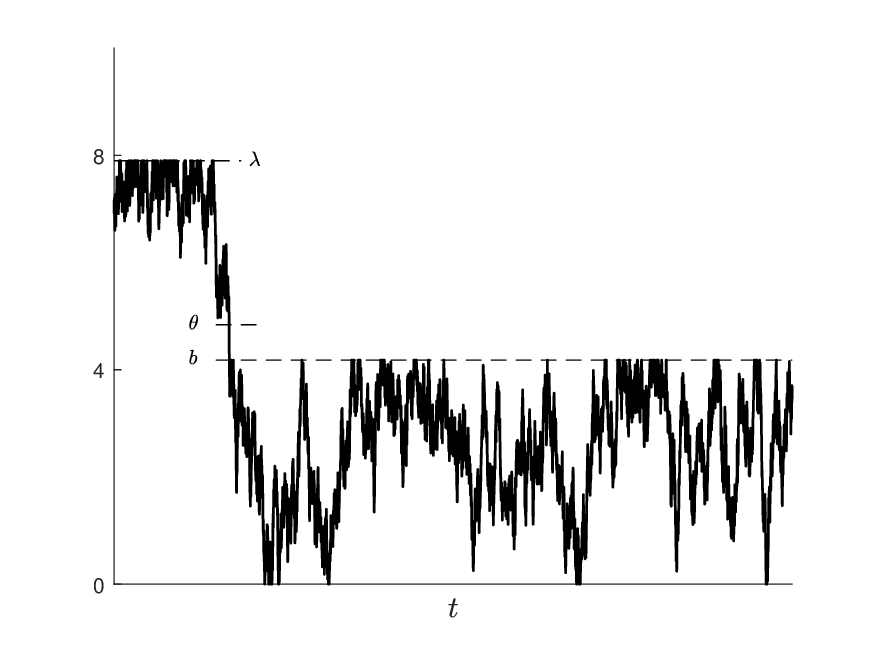} 
                 \caption{A sample path of  \eqref{eq: X}  under optimal control.}
        \label{fig: Fig1}
\end{figure}

This result, which is conjectured by Liu et al \cite{LLS03}, appears eccentric enough to call the model into question.  (Incidentally, of the people that were presented with the problem in discussion, none anticipated the stated result, and, when  informed of the result,  not one was able to provide a rationale or empiric basis for operating with three thresholds.)  Liu et al \cite{LLS03} suggest that the result is caused by the choice of marginal yield function. A  proper explanation of the result must in some way  account for the reflection, however,  because we will see  that  it is always optimal to use a single barrier in the absence of reflection, for  \emph{any} concave yield function  (see \Cref{rem: no L}). With the benefit of hindsight and the recent results of Ferrari \cite{Fer19},  we can interpret the band policy  portrayed in \Cref{fig: Fig1} as involving a switch between two modes of operation with different risk profiles. 
We postpone its discussion to our concluding remarks (\Cref{sect: discussion}).  %

The problem studied in this paper is closely related to the  problem of Karatzas and Shreve \cite{KS85}  in which   a reflected Brownian motion is controlled by a process of bounded variation.  Versions of this problem  arise as inventory and asset management problems where the last two terms in    \eqref{eq: X}   are interpreted as \enquote{withdrawals} and \enquote{deposits}, respectively \cite{TLY22,ES11,LZ08,FS19,ZY16}. The form of the optimal policy in these problems is well understood in the case when the marginal yield function is constant. In this case, the optimal policy (when one exists) involves a single upper barrier  \cite{SLG84}. Ferrari  \cite{Fer19} shows that the same holds true  under a structural assumption on the   yield   function   (see \eqref{eq: structural cond} below).   The main contribution of this paper is to show that the optimal policy can be more complicated  than  anticipated. 

 We believe that our  results will be of interest    from a methodological perspective as well as from the point of view of applications.  The  methodological  interest  lies in  the fact that the application of \emph{smooth fit} becomes complicated  when  the  structural condition from  \cite{Fer19}   is relaxed.  More precisely, we find a partial lack of  smooth fit, as the value function fails to be of class $C^2$ at the   \enquote{repelling}  $\theta$-boundary. Smooth fit has proven highly effective in solving optimal control problems  and there is a growing literature on its limitations \cite{Pes07,GW09,GT09,DFM15,DFM18}.  It is known, for instance, that    smooth fit may fail due to the non-convexity of the cost structure  \cite{DFM18}  or the irregularity of the scale function of the underlying diffusion \cite{Pes07}.    We present what to our knowledge is the first example demonstrating that a lack of smooth fit can arise as a consequence of exogenous reflection.

  Our practical interest in   understanding the phenomenon  portrayed in \Cref{fig: Fig1}   is motivated by the popularity of the reflected Brownian model \eqref{eq: X} in applications of the aforementioned types. In applications, one typically wants to  determine a single threshold at which control should be exerted. (We mention an exception to this rule  in \Cref{sect: discussion}.)  Much research effort has in fact focused on identifying conditions on the yield structure that guarantee the existence of   optimal policies of  barrier type \cite{Por77,Fer19,Zei04,Alv00,JJZ08}.  There thus seems to be disagreement between the management of real world systems on the one hand and the model's prescriptions on the other. We   believe that  this  may lead  many to wonder  (as we have)  if   the phenomenon  portrayed in Figure \ref{fig: Fig1} is signalling a flaw in the  model of how such systems evolve in time.  As indicated above, we will suggest, on the contrary, 
  that the phenomenon provides insight.   At the same time, our results indicate that  the phenomenon  is very rare and that there is little to be gained by operating with three thresholds  (see \Cref{rem: sensitivity to perturb}).  In practice,   the structural condition from  \cite{Fer19} thus seems rather unrestrictive.

The paper  is organized as follows:   In the next section we give a precise statement of the problem and of our main result. 
\Cref{sect: v for simple} describes the methods that we will use and presents some auxiliary results.   As in \cite{LLS03}, we use a \text{guess-and-verify} app\-roach relying on a martingale formulation of   the dynamic programming principle. In \Cref{sect: barrier p,sect: optimal p} we prove the main result. Finally, in \Cref{sect: discussion}    we  discuss the mode switch alluded to above.

    \section{Problem formulation and statement of results}\label{sect: problem}
 \subsection{Problem formulation} To formulate our problem in precise terms,   we take as primitive  a  Brownian motion $W=\{W_t\}_{t\geq 0}$ on a complete   probability space with a filtration $\{\mathscr{F}_t\}$ satisfying the usual assumptions.  The right-continuous, non-decreasing, $\{\mathscr{F}_t\}$-adapted processes form the class $\bA$ of   admissible control   policies. Given   ${\bk} \in \bA$, the non-decreasing process $\dL$ pushes  by the minimal amount needed   to ensure   that   \eqref{eq: X}    stays non-negative. In the absence of control, $\dL$ is just the local time  at $0$ of a Brownian motion starting at $x$.   In general,  $\dL$  may jump if needed to prevent     $X$  from jumping across the origin. However, we will primarily be interested in policies for which $\dL$ is  a.s.  continuous as a function of $t$.

In view of previous studies of similar problems (see \cite{BC67,BSW80,Kar83}), we are especially interested in policies   with  singular (local-time-like)  parts. Such a  policy is here characterized by a closed set $\mathcal{A}\subset (0, \infty)$ of \emph{action} and its complement (the region of \emph{inaction}) as follows: 
If $X$ starts  in the interior of $\mathcal{A}$, then control is exerted   to instantaneously get   to the boundary of $\mathcal{A}$;  if $X$ starts  in the complement of $\mathcal{A}$, then control is exerted by     the minimal amount   needed to prevent   $X$ from crossing the boundary of  $\mathcal{A}$.  If ${\bk}$   has action region   $[b,\infty)$,  then ${\bk}$ serves to enforce an \emph{upper} reflecting barrier at $b$.    We refer to such a policy as  a \emph{barrier policy}.  
  The  barrier policy with threshold $b$ is denoted by   ${\bk}^{(b)}$.  
 
If  ${\bk} \in \bA$  is a.s. continuous as a function of $t$,   the   total discounted   yield that ${\bk}$ generates over the interval $[s, t) \subset [0, \infty)$ is given by  the random variable  
   \begin{align}\label{def: Sti}
   \int_s^t e^{-{r} u}\eta(X_u)d{\bk}_u.
 \end{align} 
For possibly discontinuous  ${\bk}  \in \bA$, we are led to consider  
        \begin{align}\label{def: eta circ Z}
  \int_{[s,t)}e^{-{r} u}\eta(X_u) \circ d{\bk}_u:= &\int_s^t e^{-{r} u}\eta(X_u)d{\bk}^c_u 
  +\sum_{s\leq u < t}e^{-{r} u}\int_{X_s}^{X_{s_-}} \eta(u)du,
 \end{align}
 where ${\bk}^c$ is the continuous part of ${\bk}$ and where the sum is taken over the disconti\-nuity points of  ${\bk}$. The definition \eqref{def: eta circ Z} represents the standard  way of dealing with the fact that  the Stieltjes integral  \eqref{def: Sti} does not properly account for lump rewards generated at discontinuity points of ${\bk}$  \cite{DZ98,Zhu92,JJZ08}.\footnote{The right hand side of  \eqref{def: eta circ Z} actually needs an additional term if ${\bk} \in \bA$ is such that  $\dL$  must jump to prevent  $X$  from jumping across the origin. (Our definition of $\bA$  \emph{does} allow for this possibility.) However, \Cref{assumption1} below will  leave the decision maker  without incentive to exert control  near  $0$. For expositional brevity, we will  therefore ignore the possibility that $\dL$ has jumps.} If $s=0$ and ${\bk}_0>0$,   the sum in \eqref{def: eta circ Z} carries a leading term of $\int^{x}_{x-{\bk}_0}\eta(u)du$. 
 
 The \emph{value} of   ${\bk} \in \bA$ is defined   
   \begin{align} 
   V_{\bk}(x)=\mathbb{E}_x \int_{[0,\infty)} e^{-{r} t}\eta(X_t) \circ d{\bk}_t,
 \end{align}
 where $\mathbb{E}_x$ denotes expectation conditioned on the fact that $X$ starts at $x$.   The problem is to maximize $V_{\bk}(x)$ over ${\bk} \in \bA$. The \emph{value function}, $V(x)$,    
  is defined 
   \begin{align}\label{def: V}
   V(x)=\sup_{{\bk} \in \bA} V_{\bk}(x). 
 \end{align}

Our main result 
concerns the marginal yield function 
 \begin{align}\label{def: eta}
 \eta(x):=(1-\frac{1}{x})^+ 
 \end{align}
  from \cite{LLS03,Zei04}. 
 Throughout, we make the following  assumptions on $\eta$. 
   \begin{assumption}\label{assumption1} {\hphantom{h}}       \begin{itemize} 
      \item[(i)]  $\eta\colon [0, \infty) \to [0, 1]$ is  increasing and  concave   on its support, 
       \item[(ii)]  $\eta(x^\ast)=0$ for  some  $x^\ast >0$.
 \end{itemize}
  \end{assumption}
 It is easy to show that $V$ would be infinite if $\eta(0)$ were strictly positive. \Cref{assumption1}\,(ii)     removes   incentive to exert control near $0$  and ensures that  $V$ is finite. A more general condition is formulated as Assumption 2.9 in   \cite{Fer19}.   
    
\subsection{Prior results}\label{subsect: prior results} 
Aside from  \cite{LLS03}, the two works most closely related to our study  are   the papers by  Zeitouni  \cite{Zei04} and  Ferrari 
\cite{Fer19}.  Zeitouni \cite{Zei04} formulates and partially solves  the problem that we have stated. More precisely,   she identifies a subset of the parameter space on which the optimal policy involves a single upper barrier.  Ferrari \cite{Fer19}  obtains similar results for   general reflected diffusions under a condition on the yield structure. In the case of constant  coefficients, this structural condition  states that 
\begin{align}\label{eq: structural cond}
-r\cdot \eta(x)+\mu\cdot \eta'(x)+\frac{\sigma^2}{2} \eta''(x)
 \end{align} changes sign at most once    \cite[p. 953]{Fer19}. Similar conditions are used by other authors to ensure the   existence of   optimal  barrier policies (see, e.g., Theorem 2 in \cite{Alv00}, Assumption 5 in \cite{JJZ08} and Assumption 2.5  in \cite{LZ11}).  
 
We remark that  \cite{Fer19}  concerns the more general problem of maximizing   
\begin{align}\label{def: V Fer}
   \mathbb{E}_x  \Bigl( \int_{[0,\infty)} e^{-{r} t}\eta(X_t)\circ d{\bk}_t-\kappa \int_0^\infty e^{-rt}d\dL\Bigr),
 \end{align}
where $\kappa  \geq \eta(0)$ is interpreted as the marginal cost for reflection. Costly (endogenous) reflection is considered in  optimal  harvesting and renewing problems and in optimal dividends problems with compulsory capital injections  \cite{TLY22,LZ08,ZY16}. Our main result holds also with respect to the criterion \eqref{def: V Fer}, at least if $\kappa>0$ is small. (The condition $\kappa \geq \eta(0)$ holds for every $\kappa>0$  by  \Cref{assumption1}\,(ii).)  Note, however, that the assumption of \emph{increasing} marginal yield needs motivation in the optimal dividends interpretation of the problem, where it is commonly assumed that $\eta$ imposes proportional transaction   costs on dividends (e.g., in the form of a  
constant tax rate). 
\subsection{The main result}\label{subsect: main result} 
Our main result is the following.    
 
  \begin{theorem}\label{theorem1} 
  Let  $\eta$  be   as in  \eqref{def: eta}. 
 For  some   $\mu$, $\sigma$ and $r$,  there are constants $0<b<\theta<\lambda$ such that  \emph{(i)}  ${\bk}^{(b)}$ is optimal within the class of barrier policies, \emph{(ii)} the   optimal policy has action region  $[b, \theta] \cup [\lambda, \infty)$,  \emph{(iii)}  $V$ is twice continuously differentiable  on $[0, \theta)\cup (\theta, \infty)$, but   only once continuously differentiable at  $x=\theta$.
\end{theorem} 

 \section{The Bellman principle and  the verification theorem}\label{sect: v for simple}
As in \cite{LLS03},  we adopt  a guess-and-verify-app\-roach relying on the  Bellman principle of optimality. The methods that we will use are similar to those in \cite{AS98,MS05,RS96}. However, in contrast to these works, we will here need a  verification theorem for  non-$C^2$-functions (\Cref{lemma2} below).

The Bellman principle here asserts that an optimal  ${\bk}\in \bA$   must  be such that
 \begin{align}      \label{eq: Bell1}
     V(X_t)e^{-rt}=\mathbb{E}_x \Bigl( \int_{[t,\infty)} e^{-ru}\eta(X_u)\circ d{\bk}_u  \, \Big| \, \mathcal{F}_t  \Bigr) \;   \text{ for all  }t>0.
      \end{align}
    That \eqref{eq: Bell1}  is satisfied means that  
                \begin{align}      \label{def: Y0}
      Y_t:=V(X_t)e^{-rt}+\int_{[0,t)}e^{-rs}\eta(X_s)\circ d{\bk}_s, \; t \geq 0,
 \end{align}
  is  a  $\{\mathscr{F}_t\}$-martingale.  (The latter property of  \eqref{def: Y0}    is actually equivalent to  \eqref{eq: Bell1}.) 
To see this, note that    if  \eqref{eq: Bell1}  holds, then  \eqref{def: Y0}  reads as         
  \begin{align}    \label{def: Y02}
     Y_t=\mathbb{E}_x \Bigl( \int_{[0,\infty)} e^{-ru}\eta(X_u)\circ d{\bk}_u \, \Big| \, \mathcal{F}_t \Bigr),  \; t> 0. 
           \end{align}
By the tower property of conditional expectation,  
\eqref{def: Y02}    implies that  
$\mathbb{E}_x( Y_t \, \big| \,  \mathcal{F}_s)= Y_s$ for $s<t$. Hence, if    \eqref{eq: Bell1} holds,  then  \eqref{def: Y0}  is a martingale.

The Bellman principle thus says  that our search for $V$ and   optimal $\bk $   can   be restricted to pairs $(v,  {\bk})$ for which 
     \begin{align}      \label{def: Y}
      Y^{v, {\bk}}_t:=v(X_t)e^{-rt}+\int_{[0,t)}e^{-rs}\eta(X_s)\circ d{\bk}_s, \; t \geq 0\;  (Y^{v, {\bk}}_0:=v(x)),
 \end{align}
 is a  martingale. The following result tells us how $v$ should be constructed  for this to be so.
 
 \begin{lemma} \label{lemma1} Suppose that $v \colon [0, \infty) \to \mathbb{R}$ has an absolutely continuous derivative. For ${\bk} \in \bA$,  let  $Y^{v, {\bk}}$ be defined as in  \eqref{def: Y}.   Then
        \begin{align}
Y^{v, {\bk}}_t&=Y_0+\sigma \int_0^t e^{-rs}dW_s+\int_0^t v'(X_s)d\dL_s+\int_0^t \bL v(X_s)e^{-rs}ds \label{eq: Ito 1} \\&\sum_{0<s\leq t} e^{-rs}\int^{X_{s}}_{X_{s-}} \bM v(u)du+ \int_0^t\bM v(X_s)d{\bk}_s^c, 
\label{eq: Ito 2}
 \end{align}
where the sum in  \eqref{eq: Ito 2} are taken over the disconti\-nuity points of  $X$ and where
 \begin{align}
\bL v(x):= & -r v(x)+\mu\cdot v'(x)+\frac{\sigma^2}{2} v''(x), \label{eq: Lv} \\
\bM v(x):=&\eta(x)-v'(x).
 \end{align}
       \end{lemma}
 From  \Cref{lemma1} we see that $Y^{v, {\bk}}$ is a martingale if all but the first two terms in \eqref{eq: Ito 1}-\eqref{eq: Ito 2}  are $0$. In the case when {\bk} is singular with action region $\mathcal{A}\subset (0, \infty)$,      $\dL$ is   continuous and flat off $\{t\colon X_t=0\}$. Hence,  the third term in \eqref{eq: Ito 1}   is $0$ if   $v'(0)=0$.  The set $\{t\colon X_t \in \mathcal{A}\}$ has Lebesgue measure zero,  so  the last intergal in \eqref{eq: Ito 1}  is $0$ if $\bM v$ vanishes on $\mathcal{A}$. The terms in \eqref{eq: Ito 2}  are both $0$ if $\bM v$ vanishes \emph{off}  $\mathcal{A}$. Thus,  $Y^{v, {\bk}}$ is a martingale if  $v$ 
 is  such that 
       \begin{align}\label{eq: Mv Lv zero}
\bM v(x)=0  \text{ for all }x \in \mathcal{A}  \; \text{ and }\; \bL v(x)=0  \text{ for all }x \in \mathcal{A}^c. 
 \end{align} 

\begin{proof}[Proof of \Cref{lemma1}]
  The   process  \eqref{eq: X}  is a right-continuous semimartingale with stocha\-stic  differential 
     \begin{align}\label{eq: dX post}
dX_t=\mu t+\sigma dW_t+d\dL_t-d{\bk}_t.
 \end{align}
That $v$ has an absolutely continuous derivative means that we can apply the It{\^{o}}-Meyer formula (see \cite[Theorem 71]{Pro05})  to $v(X)$. We then get
\begin{align}
v(X_t)=&v(X_0)+\int_0^t v'(X_{s-})dX_s+\frac{\sigma^2}{2}\int_0^t v''(X_{s-})d[X, X]_s^c \notag \\
&+\sum_{0<s\leq t}\{v(X_s)-v(X_{s-})-v'(X_{s-})\triangle X_s\}.
\label{eq: I-M}
 \end{align}
Here, $[X, X]$ is the quadratic variation of $X$,   sums are taken over discontinuity points of $X_s, s \in (0, t]$, and $\triangle X_s:=X_s-X_{s-}$.  Since ${\bk}$ is non-decreasing, $[X, X]_s=\sigma^2\cdot s$. 
The discontinuity points of $X_s, s \in (0, t]$,  are also discontinuity points of ${\bk}_s, s \in (0, t]$,  where $\triangle X_s=- \triangle {\bk}_s:=-({\bk}_s-{\bk}_{s-})$. Hence, 
      \begin{align*}
 -\sum_{0<s\leq t}v'(X_{s-}) \triangle X_s&= \sum_{0<s\leq t}v'(X_{s-})\triangle {\bk}_{s}=\int_0^t v'(X_{s-})d{\bk}_s-\int_0^t v'(X_{s-})d{\bk}_s^c.
 \end{align*}
Writing $v(X_s)-v(X_{s-})=-\int^{X_{s-}}_{X_s}v'(s)ds$ and collecting terms in \eqref{eq: I-M} gives 
   \begin{align}\label{eq: IM2}
v(X_t)=&v(X_0)+\sigma  \int_0^t e^{-rs}dW_s + \mu \int_0^t  v'(X_{s-})ds +\frac{\sigma^2}{2}\int_0^tv''(X_{s-})ds\notag \\
&-\sum_{0<s\leq t}\int^{X_{s-}}_{X_s}v'(s)ds-\int_0^t v'(X_{s-})d{\bk}_s^c.
 \end{align}
The corresponding representation for  the product $v(X_t)e^{-rt}$ is a straightforward consequence of  \eqref{eq: IM2} 
  (cf.  \cite{EK22}). We get   
       \begin{align}
v(X_t)e^{-rt}=&v(X_0)+\sigma  \int_0^t dW_s +   \int_0^t(-r\cdot v(X_s)+\mu \cdot v'(X_{s-}) +\frac{\sigma^2}{2}v''(X_{s-}))e^{-rs}ds\notag \\
&-\sum_{0<s\leq t}\int^{X_{s-}}_{X_s}v'(s)e^{-rs}ds-\int_0^t v'(X_{s-})e^{-rs}d{\bk}_s^c.
\label{eq: I-M last}
 \end{align}
Adding  
    \begin{align*}
\int_0^te^{-rs} \eta(X_s)\circ d{\bk}_s=\int_0^te^{-rs} \eta(X_s)d{\bk}_s^c +\sum_{0 \leq s\leq t} \int^{X_{s-}}_{X_s}e^{-rs}\eta(s)ds
 \end{align*}
 to \eqref{eq: I-M last} obtains \eqref{eq: Ito 1}-\eqref{eq: Ito 2}.
    \end{proof}
 
The verification-part of our   strategy for finding $V$  relies on the following result. 
   \begin{lemma} \label{lemma2} Suppose that $v \colon [0, \infty) \to \mathbb{R}$ has an absolutely continuous derivative, sub-exponential growth,     and that $v'(0)=0$. If, for all  $x \in [0, \infty)$, we have 
       \begin{align} 
    \bL v(x) \leq 0  
    \; \text{ and }  \;  \bM v(x)\leq 0, 
       \label{eq: Lv Mv leq}
 \end{align} 
then $v(x)\geq V(x)$,   for  all $x \in [0, \infty)$. 
  \end{lemma}
       \begin{proof} Let ${\bk} \in \bA$. By \Cref{lemma1},  the   inequalities in   \eqref{eq: Lv Mv leq} imply that  the process $Y^{v, {\bk}}$ is a    ${\mathscr{F}}_t$-supermartingale.  Thus, 
 $\mathbb{E}_x(Y^{v, {\bk}}_t)\leq Y^{v, {\bk}}_0=v(x)$  for  all  $x, t \in [0, \infty)$.  
             Since   $\mathbb{E}_x(Y^{v, {\bk}}_t) \to V_{{\bk}}(x)$ as $t \to \infty$,  this means that $v(x) \geq V_{\bk}(x)$ for all $x \in [0, \infty)$ and hence, since ${\bk} \in \bA$ was arbitrary,  that $v(x) \geq V(x)$ for all $x \in [0, \infty)$. 
        \end{proof}

\section{Finding the best barrier policy}\label{sect: barrier p}
The action region in \Cref{theorem1}  involves  the value of $b$ that makes  ${\bk}^{(b)}$  optimal within the class of  barrier policies.   We begin by finding this value  using  the strategy of \cite{Alv00}, where the problem of finding the best barrier policy is treated as an optimization problem in a single real variable. 

Let us first  calculate the value of ${\bk}^{(b)}$, the barrier policy with threshold $b>0$.  The action region of ${\bk}^{(b)}$ is   the interval $[b, \infty)$. In view of  \eqref{eq: Mv Lv zero},  a  candidate $v$ for  $V_{{\bk}^{(b)}}$ should satisfy      \begin{align}   \label{eq: vp zero}
   v'(0)=0, \; 
\bL v(x)=0 \; \text{ for all   }x\in [0,  b), \; 
\bM v(x)=0 \;  \text{ for  all }x\in (b, \infty).  
 \end{align}
The general solution to \eqref{eq: vp zero} can be written  
     \begin{align*} 
v(x)= \begin{cases} &A \phi(x), \quad x \in [0,  b), \\
&B +\int_b^x\eta(u)du,  \quad x \in (b,  \infty),
 \end{cases}
 \end{align*}
 where 
     \begin{align*} 
\phi(x):= \frac{e^{\gamma_+ x}}{\gamma_+}- \frac{e^{\gamma_- x}}{\gamma_-}, \; x \in [0,  1], \quad \gamma_\pm:=\frac{-\mu \pm \sqrt{\mu^2+2r\sigma^2}}{\sigma^{2}}.
  \end{align*}
For $v'$ to be (absolutely) continuous,  
it is both necessary and sufficient that 
   \begin{align}\label{def: A and B}
B=A\phi(b),  \qquad \eta(b)=A\phi'(b). 
  \end{align}
Solving for $A$ and $B$ gives 
    \begin{align}\label{def: V-Z}
v(x)=\begin{cases}
\frac{\eta(b)}{\phi'(b)}\phi(x),  \;  x \in [0, b), \\
\frac{\eta(b)\phi(b)}{\phi'(b)}+  \int_{b}^x \eta(u)du, \;   x \in [b,  \infty).
\end{cases}
 \end{align}
 With this $v$,  $Y^{v, {{\bk}^{(b)}}}$ is a  martingale by   \Cref{lemma1}, so $\mathbb{E}_x(Y^{v, {{\bk}^{(b)}}}_t)=Y^{v, {{\bk}^{(b)}}}_0=v(x)$ for all $x, t \in [0, \infty)$. Since $\mathbb{E}_xY^{v, {{\bk}^{(b)}}}_t \to V_{{\bk}^{(b)}}(x)$ as $t \to \infty$, this means that $v \equiv V_{{\bk}^{(b)}}$.  
 
Now consider the problem of maximizing $V_{{\bk}^{(b)}}(x)$ 
for fixed $x$.  The derivative of  \eqref{def: V-Z} with respect to $b$  is
   \begin{align}\label{eq: derivative}
\frac{d V_{{\bk}^{(b)}}(x)}{db}=\begin{cases}& \frac{\phi(x)}{\phi'(b)^2}\bigl( \eta'(b)\phi'(b) - \eta(b)\phi''(b)\bigr), \; x \in [0, b)\\
&  \frac{\phi(b)}{\phi'(b)^2}\bigl( \eta'(b)\phi'(b) - \eta(b)\phi''(b)\bigr), \;  x \in [b,  \infty).
 \end{cases}
 \end{align}
Thus, for each fixed  $x \in [0, 1]$, 
$b$ is a critical point of  the function 
$b \mapsto V_{{\bk}^{(b)}}(x)$    if and only if    
 \begin{align}\label{eq: critical level2}
\frac{\eta(b)\phi''(b)}{\phi'(b)} =\eta'(b).
\end{align}
We  get the same equation  using smooth fit, which here amounts to imposing $C^2$-smoothness on $V_{{\bk}^{(b)}}$.  For the left- and (respectively) righthand  side of \eqref{eq: critical level2} give the left and  right second derivatives of \eqref{def: V-Z} at $x=b$.  Hence,   $V_{{\bk}^{(b)}}$ (now considered as a function of $x$) is of class $C^2$  if and only if $b$ solves \eqref{eq: critical level2}.  
 
For \enquote{most} values of the three parameters, \eqref{eq: critical level2} has a single root $b^\ast> 0$ and  ${{\bk}^{(b^\ast)}}$ is optimal.  However,   \eqref{eq: critical level2}  \emph{can} have  more than one root. To see this, let $\eta$  be defined as in  \eqref{def: eta}. Then  \eqref{eq: critical level2} can be written 
      \begin{align}\label{eq: critical level3}
 b^2-b= \frac{e^{\gamma_+ b}-e^{\gamma_- b}}{\gamma_+e^{\gamma_+ b}-\gamma_-e^{\gamma_- b}}.
  \end{align}
Fixing $\sigma$ and choosing $\mu$ and $r$ using a random number generator, one eventually finds values of $\sigma$, $\mu$ and $r$ for which \eqref{eq: critical level3} has  three roots.  
We  focus on just one such case: 
   \begin{align}\label{eq: parameters}
\sigma= \sqrt{2}, \quad \mu =0.508378, \quad r= 0.00520074.
\end{align}
For these parameters,  the three roots of the equation  \eqref{eq: critical level3}  are 
 \begin{align}\label{eq: three b}
 b_1 \approx 4.18138,  \quad b_2 \approx5.760862,  \quad b_3 \approx8.003166. 
 \end{align}	
By our previous observation concerning  \eqref{eq: critical level2},  
each $b_i$ is a critical point of $b\mapsto V_{{\bk}^{(b)}}$. As \Cref{fig:Vb} shows,  $V_{{\bk}^{(b)}}$ has a local minimum at $b_2$ and local maxima at  $b_1$ and $b_3$.\footnote{\Cref{fig:Vb}  displays the graph of  $V_{{\bk}^{(b)}}(5)$. That the shape of the graph of $V_{{\bk}^{(b)}}(x)$ does not depend on $x$ is due to the fact that the  sign of \eqref{eq: derivative} does not depend on $x$.}  The global maximum is attained at $b=b_1$, so ${{\bk}^{(b_1)}}$ is optimal within the class of barrier policies.   In particular,  ${{\bk}^{(b_1)}}$   slightly outperforms ${{\bk}^{(b_3)}}$.

 \begin{figure}[h]
     \centering
         \begin{subfigure}[b]{0.53\textwidth}
         \centering
    \includegraphics[width=0.8\textwidth]{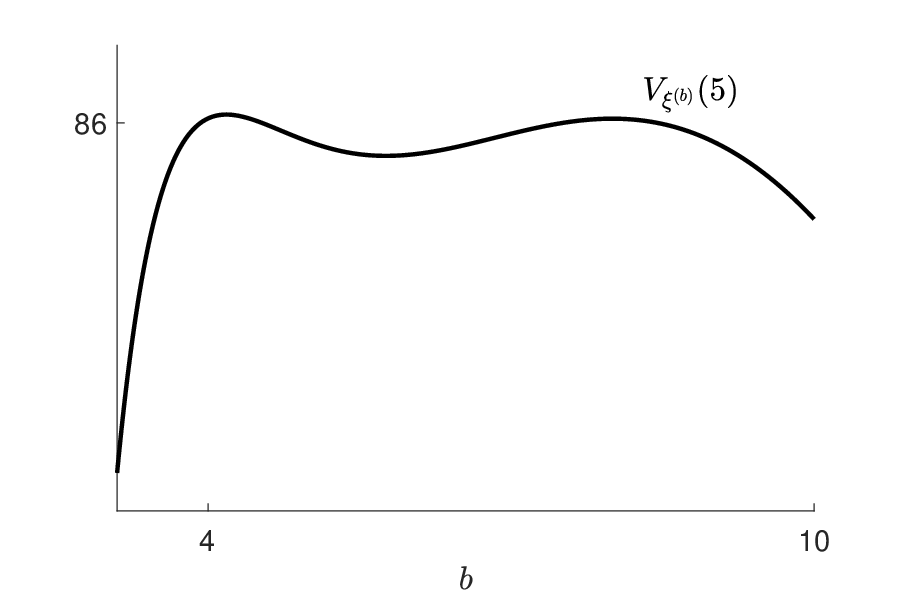}
         \caption{$V_{{\bk}^{(b)}}(5)$ vs $b$.}
         \label{fig:Vb}
     \end{subfigure}	
     \hfill
       \begin{subfigure}[b]{0.46\textwidth}
       \includegraphics[width=0.9\textwidth]{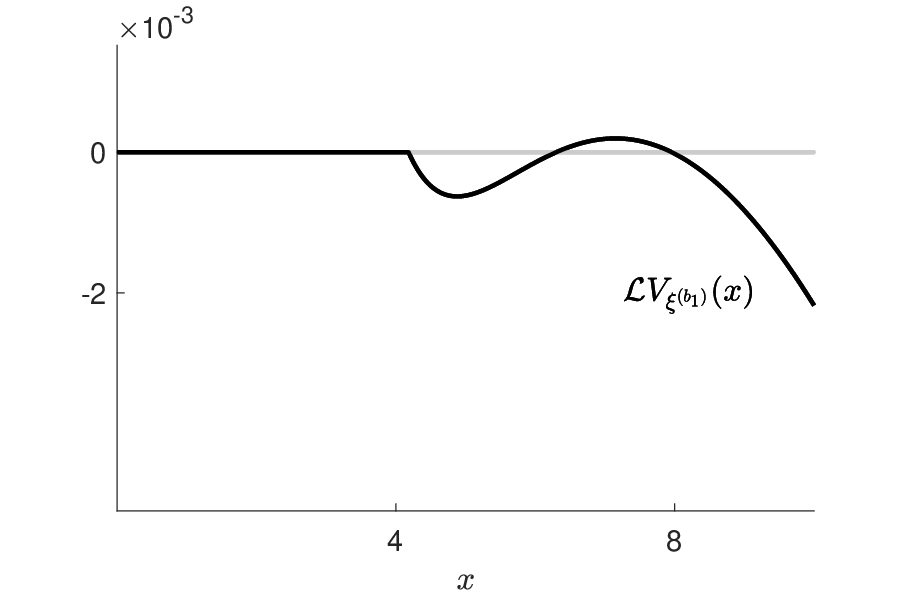}
         \caption{$\bL V_{{\bk}^{(b_1)}}(x)$ vs $x$.}
              \label{fig:MV and AV1}
     \end{subfigure}
        \caption{  Graphs of  $V_{{\bk}^{(b)}}(5)$ and $\bL V_{{\bk}^{(b_1)}}(x)$ }
        \label{fig:MV and AV}
\end{figure}

To show that  ${{\bk}^{(b_1)}}$ is optimal, and not just the best barrier policy,  we would need to complete the verification step. But, as \Cref{fig:MV and AV1} shows,    $V_{{\bk}^{(b_1)}}(x)$   does not satisfy    \eqref{eq: Lv Mv leq}. (The same is true of $V_{{\bk}^{(b_2)}}$ and $V_{{\bk}^{(b_3)}}$.) We must therefore seek a new candidate for $V$. 
 
 \begin{remark}\label{remark: three roots}
We emphasize that   multiple roots of  \eqref{eq: critical level3}   does \emph{not} rule out the existence of an optimal barrier policy. In fact, in most cases that  we considered where \eqref{eq: critical level3} has three roots,  the best barrier policy is optimal in $\bA$ and smooth fit holds. As we have already noted,  the roots of \eqref{eq: critical level3} are precisely the critical points of $b\mapsto V_{{\bk}^{(b_1)}}$. 
In general, 
if \eqref{eq: critical level3} has three roots,  then  $b\mapsto V_{{\bk}^{(b_1)}}$  has   local maxima at the smallest and the largest of the three roots (compare \Cref{fig:Vb}).  These maxima are related to the  \enquote{two modes}  alluded to above.  
  \end{remark}
  
   \begin{remark}\label{remark: three roots2}
  Liu et al \cite{LLS03} show that 
  \eqref{eq: critical level3}   has  \emph{at most one} root if $e^{\gamma_--\gamma_+}<-\gamma_+/\gamma_-$.  They also show that this equation has \emph{at most  three} roots.  
 We have not been able to find a yield function (satisfying \Cref{assumption1}) for which the corresponding equation  \eqref{eq: critical level2} has  more than three roots. 
  \end{remark}

\section{Completing the proof of \Cref{theorem1}}
\label{sect: optimal p}
  Let  $\eta$  be the function in  \eqref{def: eta}, and let $\mu, \sigma$ and $r$  be as  in  \eqref{eq: parameters}. We aim  to show that the optimal policy has action region $\mathcal{A}=[b_1, \theta] \cup [\lambda,  \infty)$, where $b_1$ is  the smallest of the three roots of the equation  \eqref{eq: critical level3} for a smooth fit.  In view of \eqref{eq: Mv Lv zero}, we define our candidate for $V$ as  
   \begin{align}\label{def: V anom}
v(x)=\begin{cases}
A(\frac{e^{\gamma_+ x}}{\gamma_+} -\frac{e^{\gamma_- x}}{\gamma_-}), \; &x \in [0,  b_1], \\
B+\int_{b_1}^x\eta_0(u)du, \;   & x \in (b_1, \theta],\\
C_1 e^{\gamma_+ x}+C_2 e^{\gamma_- x}, \; &x \in ( \theta, \lambda], \\
D +\int_{\lambda}^x\eta_0(u)du, \;  &x \in ( \lambda, \infty).\\
\end{cases}
 \end{align}

\noindent Here, $A$ and $B$ are defined by $b_1$  and the relations   \eqref{def: V-Z}.  We let the remaining five constants   $(C_1, C_2, D, \theta$ and  $\lambda)$ be determined by   $C^2$-smoothness at $x=\lambda$ and $C^1$-smoothness at $x=\theta$. These smoothness requirements lead to the equations 
 \begin{align} 
C_1 e^{\gamma_+ \theta}+C_2 e^{\gamma_- \theta} -(B+\theta-b_1-\log(\theta/b_1))&=0, \notag \\
C_1 \gamma_+e^{\gamma_+ \theta}+C_2\gamma_-e^{\gamma_- \theta}-(1-1/\theta)&=0, \notag \\
C_1    e^{\gamma_+ \lambda}+C_2  e^{\gamma_- \lambda}-D&=0, \label{eq: system of equations}\\
C_1 \gamma_+e^{\gamma_+ \lambda}+C_2\gamma_-e^{\gamma_- \lambda}-(1-\frac{1}{\lambda})&=0, \notag \\
C_1 \gamma_+^2e^{\gamma_+ \lambda}+C_2\gamma_-^2e^{\gamma_- \lambda}-\frac{1}{\lambda^2}&=0. \notag
 \end{align}
To verify that $V$ is indeed on the form \eqref{def: V anom},  we   must  first specify \eqref{def: V anom}, that is, 
 \begin{align} \notag
&\text{ (i) provide a positive solution  $(C_1, C_2, D, \theta, \lambda)$ to  \eqref{eq: system of equations}  with  $b_1<\theta <\lambda$,}   
\intertext{and  then}
&\text{  (ii)  verify  the  inequalities \eqref{eq: Lv Mv leq}  when  \eqref{def: V anom} is specified by  $(b_1, C_1, C_2, D, \theta, \lambda)$. } 
 \end{align}
We only found one positive solution  to   \eqref{eq: system of equations} with  $b_1<\theta <\lambda$:\footnote{\label{foot: fsolve}Using MatLab's \texttt{fsolve} with $b_1=4.181380317$ and the initial guess $[7,1,8, 6, 8]$ obtains  $C_1 = 8.168984282, C_2 =  1.586783674, D = 8.84977837060, \theta = 4.848847551,  \lambda = 7.950177923$. } 
 \begin{align}\label{sol: system of equations} 
(C_1, C_2, D, \theta, \lambda) \approx  (0.817, \, 1.58, \, 0.885,  \, 4.84, \, 7.95). 
\end{align}
Now use this solution to define $v^\ast$ according to \eqref{def: V anom}, and   let ${\bk}^\ast$ be the singular  policy with  action region $\mathcal{A}=[b_1, \theta] \cup [\lambda, \infty)$. Then $v^\ast$  meets the requirements in \Cref{lemma1} and satisfies \eqref{eq: Mv Lv zero}, so  $Y^{v^\ast, {\bk}^\ast}$ is a martingale. Conclude that $v^\ast=V_{{\bk}^\ast}$. 

 Turning to the verification step, \Cref{fig:MV and AV opt} displays the graphs of $\bL v^\ast$ and $\bM v^\ast$. (The discontinuity of  $\bL v^\ast$ at $x=\theta$ is a consequence of the discontinuity of the second deri\-va\-tive of $v^\ast$ at $x=\theta$.)    We see that the  inequalities in  \eqref{eq: Lv Mv leq} do hold. Since $v^\ast$ satisfies the conditions in \Cref{lemma2}, this means that $v^\ast (x) \geq V(x)$, for all $x\geq 0$. We can now conclude that  $v^\ast= V$, so that ${\bk}^\ast$ is optimal, and the proof of \Cref{theorem1} is therefore complete. 
 \begin{figure}[h]
     \begin{subfigure}[b]{0.47\textwidth}
         \centering
     \includegraphics[width=0.9\textwidth]{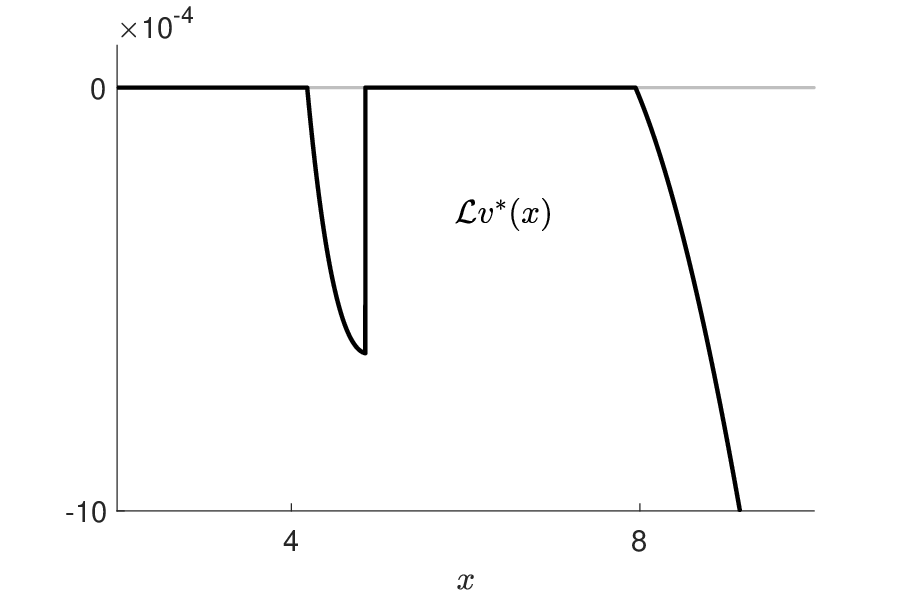}
         \caption{$\bL v^\ast$}
         \label{fig: LVast}
     \end{subfigure}	
      \begin{subfigure}[b]{0.47\textwidth}
         \centering
     \includegraphics[width=0.9\textwidth]{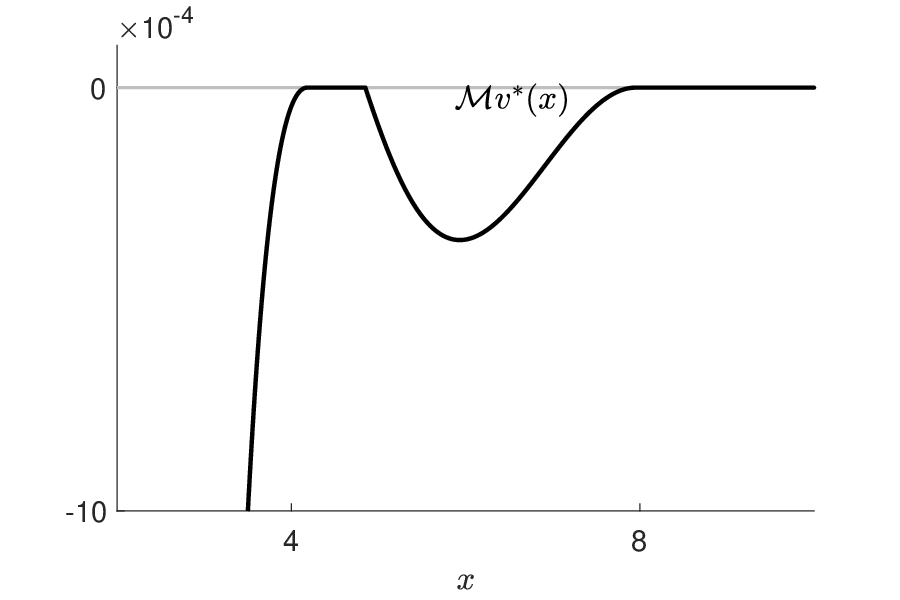}
         \caption{$\bM v^\ast$}
         \label{fig: MVast}
     \end{subfigure}	
 
        \caption{Graphs of  $\bL v^\ast$ and  $\bM v^\ast$.} 
        \label{fig:MV and AV opt}
\end{figure}

 \begin{remark} 
We have not  rigorously  proved that \eqref{eq: critical level3}  and \eqref{eq: system of equations}      have solutions  with the   properties in (i)-(ii). However, it is easy to check that such a solution exists (see footnote \ref{foot: fsolve}).  The  difficult part of the proof was coming up with the correct guess that $V$ is merely \emph{once} continuously differentiable    at  $\theta$. Here we have borrowed heavily from \cite{LLS03}, where this result is conjectured  without a verification argument for non-$C^2$-functions.
 \end{remark} 
 
\begin{remark} 
Note that the agent makes \emph{at most one} switch  between $\lambda$ and $b_1$  when ${\bk}^\ast$ is used: once $X$ drops to $\theta$, a barrier is enforced at  $b_1$.  In the terminology of \cite{DZ94},  $\theta$ is   \enquote{repelling}   for the optimally controlled process. It is worth noting that the breakdown of smooth fit at the repelling  boundary  is consistent with  \cite{DFM15,DFM18}.   In retrospect, the mere $C^1$-smoothness of $V$ at  $\theta$ is also in agreement with \cite{DZ94}. Indeed, the solution to our problem may   be viewed as the solution to a problem involving   control \emph{and}  stopping, where $\theta$ is the optimal stopping boundary. 
 \end{remark} 
 \begin{remark} \label{rem: sensitivity to perturb}
In the case  \eqref{eq: parameters}  that we considered, the  form of the optimal    policy is  quite sensitive to perturbations of the  parameters.  If   $\mu$  is slightly decreased,   then  \eqref{eq: critical level3} has a unique solution $b  \approx 4$ and (it turns out) the optimal policy is a barrier at this threshold. If    $\mu$ is slightly \emph{increased}, the optimal policy is a barrier at $b \approx 8$.  (In view of this and the continuity of  $v^\ast$, it  should come as no surprise that   the   advantage of using ${\bk}^\ast$ over ${{\bk}^{(b_1)}}$ is less than $0.01\,\%$.) Informally, the phenomenon   that we have encountered  arises on the boundary between two parts of the parameter space on which the optimal policies have distinct qualitative properties.  We elaborate   this interpretation in the next and final section of the paper. 
  \end{remark} 
\begin{remark}\label{rem: no L}  The phenomenon that we have encountered does not arise in the absence of reflection (i.e., if the reflection process $\dL$  is removed from \eqref{eq: X}).  
 The analysis  of this much simpler problem 
 is similar to that of the problem that we have studied, but the condition $v'(0)=0$ is no longer needed to ensure that     \eqref{def: Y} is a martingale.  For any $\eta$ satisfying  \Cref{assumption1}, the value of  ${\bk}^{(b)}$ is    given by (compare  \eqref{def: V-Z})  
\begin{align}\label{def: v no reflection}
\widetilde{V}_{{\bk}^{(b)}}(x)= \begin{cases} &  \gamma_+^{-1}   \eta(b)   \cdot e^{\gamma_+ (x-b)},  \quad x \in (-\infty, b], \\
& \gamma_+^{-1}  \eta(b)+\int_{b}^x\eta(u)du, \quad x \in (b,  \infty).
 \end{cases}
 \end{align}  
Smooth ($C^2$) fit at $x=b$ leads to  the   equation  
\begin{align}\label{eq: b}
 \gamma_+\cdot \eta(b) =\eta'(b).
 \end{align}
  By \Cref{assumption1}(i), the function $b \mapsto  \gamma_+  \eta(b) -\eta'(b)$ is strictly increasing  on $(0, \infty)$, so \eqref{eq: b} has  a unique solution  $b^\ast \in (0, \infty)$.  Using concavity of $\eta$, it is not difficult to prove that ${\bk}^{(b^\ast)}$ is optimal  in {\bA} (see \cite{Jon08}). Hence, in the absence of reflection, the value function is  of class $C^2$ and there is  always an optimal    barrier policy. 
 \end{remark}
 
   \section{Final remarks on the optimal policy}\label{sect: discussion}
   It remains to  address the question of whether  the phenomenon portrayed in \Cref{fig: Fig1} should be viewed as a flaw of the reflected Brownian model \eqref{eq: X} in applications of  the type   that we have mentioned. We will suggest that this phenomenon does in fact correspond to observed  beha\-vior,  for which the model may thus provide a rational explanation.  This interpretation should be compared with Henderson and Hobson's \cite{HH08} interpretation of a phenomenon which at first appears similarly puzzling.
   
Let us first note that although the band policy that we have encountered appears to be a new phenomenon, policies involving multiple thresholds are well known to arise in some  problems of optimal control.  One such problem is the stochastic cash balance problem studied by Harrison et al \cite{HST83} in which the contents of a cash fund can be decreased or increased at proportional \emph{plus fixed} costs. Under  optimal  control, the  fund's manager effectuates an upward jump to $q>0$ each time $0$ is hit, and  a downward jump to $Q>q$ whenever  the contents process  reaches  $S>Q$.  As noted in  \cite{HST83}, this form of the optimal policy is easily  anticipated in view of the fixed transaction costs. When these costs tend to $0$,  the optimal  policy enforces a single upper  barrier.  The   band structure of the optimal policy  is thus a consequence of features of the problem that  are    not present in our model.

 The band policy portrayed in \Cref{fig: Fig1} is best understood as a switch between two  modes of operation.  The     modes in question are related to local maxima of   $b \mapsto V_{{\bk}^{(b)}}$.   As  previously noted (see \Cref{remark: three roots}), this function has two  local maxima  if the  equation \eqref{eq: critical level3}   for a smooth fit  has three roots,    the maxima occurring at the smallest and the largest root.  These roots may be thought of as providing the best representatives from each of two qualitatively different types of barrier policies.  Policies of the first type exert control where   marginal yield is high. These policies are   risky in the sense that there may be long time intervals during which no   yield is accumulated. The second type of policies   accumulate yield closer to the origin and are therefore less risky in this sense. In the case   \eqref{eq: parameters}  that we have studied, the optimal policy has an interesting interpretation:  The risk-neutral decision maker initially gambles on  the more risky strategy, but lowers risk if this strategy underperforms.  

Such  behavior can be observed  in roughly the following situation: An agent may choose between two plans, $A$ and $B$, where $A$ has the potential  of   yielding high marginal  reward while $B$ offers lower marginal reward at lower risk.  Here, the idea of falling back on       \emph{B} if \emph{A} underperforms is quite familiar.  Our results  suggest that such behavior can align with risk-neutral preferences if the expected reward of choosing \emph{A}   is close to that of   choosing  \emph{B}.  Additionally, the expected reward of choosing \emph{A}  should be \emph{smaller} than that of  choosing \emph{B}; an  expectation-maximizing agent  would not have an incentive to abandon    \emph{A}  if the opposite were true.
\bigskip
 
\subsection*{Acknowledgement} The problem dealt with in this paper was presented to me by Larry Shepp  on my arrival to Rutgers University as a visiting student from KTH on July 5th, 2003. I am grateful for that highly stimulating   summer, during which some of the results  presented in this paper were obtained. I am also grateful to  
Naomi  Zeitouni  and Ofer Zeitouni for helpful discussions, in 2003 and in 2023. %
\bibliographystyle{plain}
\bibliography{references-Barriers}
  \end{document}